\documentclass{amsart}

\usepackage{graphicx}
\usepackage[latin1]{inputenc}

\newtheorem{thm}{Theorem}[section]

\newtheorem{prop}[thm]{Proposition}
\newtheorem{exa}[thm]{Example}
\newtheorem{dfn}[thm]{Definition}
\newtheorem{rem}[thm]{Remark}
\numberwithin{equation}{section}

\begin{document}

\title[]{On groups $G_{n}^{3}$ and imaginary generators}

\author{S.Kim}

\address{Department of Fundamental Sciences, Bauman Moscow State Technical University, Moscow, Russia \\
ksj19891120@gmail.com}

\author{V.O.Manturov}

\address{Chelyabinsk State University and Bauman Moscow State Technical University, Moscow, Russia \\
vomanturov@yandex.ru}

\begin{abstract}

In the present paper, we construct a monomorphism from (Artin) pure braid group $PB_{n}$ into a group, which is `bigger' than $PB_{n}$. Roughly speaking, this mapping is defined on words of braids by adding `new generators' between generators of $PB_{n}$. By this mapping we can get a new invariant for classical braids. As one of application of this invariant, we will show examples, which are minimal words in $PB_{n}$ and the minimality can be shown by the invariant. \\

MSC 57M25, 57M27.
\end{abstract}

\maketitle

\section{Introduction}
In the papers \cite{Manturov,ManturovNikonov}, the autors initiated the study of groups  $G_{n}^{k}$, describing the behaviour of dynamical systems of  $n$ particles governed by a certain general position codimension $1$ property with respect to $k$-tuples of particles. The main examples coming from this theory are homomorphisms from the $n$-strand pure braid group to the groups $G_{n}^{3}$ and $G_{n}^{4}$.
It turns out that the standard Artin presentation can also be described as a modification of such a standard presentation of (some subgroup of)
 $G_{n}^{2}$.
Stadard presentations of groups $G_{n}^{k}$ have a nice property: each letter in each word of this presentation has lots of ``invariants'': when applying each of the relations we either get a bijection between similar letters (in this case the corresponding letters have the same invariants), or we get a cancellation of two letters (in this case these two letters have the same values of the invariant). This leads to two obvious invariants.
 
First,for a given word  $\beta$ from  $G_{n}^{k}$ and for each generator $a_{m}$ in $\beta$ we get `indices'' $i_{a_{m}}(i)$ valued in the free product of  $(k-1)(n-k)$ copies of the group  $\mathbb{Z}_{2}$ for each $i \in \{1, \cdots n\} \backslash m$; roughly speaking, we get an invariant by counting the number of generators $a_{m'}$, which occur before  $a_{m}$ in the word $\beta$, for each $m' \subset \{1, \cdots n\} $, such that $|m' \cap m| =k-1$ and $i \in m'$.

On the other hand, by using these indices one can construct a homomorphism from the group $G_{n}^{k}$ to a free product of the groups $\mathbb{Z}_{2}$. The non-triviality of the image for this homomorphism can be easily checked and provides a sufficient conditon for the initial braid to be non-trivial.
In the present paper, by using a simple intutive construction we show how to associatie with any word in Artin's generator a certain word in a larger set of generators containing the initial word inside (Theorem~\ref{homo_tildeGn3}).
In other words, we get a phenomenon of  {\em imaginary generators} which allows one to  {\em read between letters}: for a word in letters $\sigma_{ij}$ we can ``see''  the letters $a_{ijk}$, placed between $\sigma_{ij}$, so tat the equivalence of the initial words in $\sigma_{ij}$ leads to the equivalence of the resulting words in  $\sigma_{ij},a_{ijk}$, see Definiton \ref{relations}. In the algebraic language this is described by means of an injection of a smaller group to a larger one; the compostion of this homomorphism with the obvious projection is the identity homomorphism.
This allows one to use the larger group (a small modificaton of  the group $G_{n}^{3}$) as a modification of the small group  (actualy, the braid group); which makes it possible to use  $G_{n}^{3}$ for constructing invariants of crossings of a classical braid: with each classical crossing we associate a set of invarians which do not change under the third Reidemeister move; if we apply the second Reidemeister move, the crossings which can cancel have the same invariants.
 
This allows us to ``foresee'' that some two crossings of a classical braid can not be cancelled. In Section 2, we shall define the group $\widetilde{G}_{n}^{3}$, which appears when we splice ``classical braid generators'' into the group $G_{n}^{3}$. In the Setion 3, we construct a new invariant of pure braids by using   $\widetilde{G}_{n}^{3}$ by means of a homomorphism from  $G_{n}^{3}$ to the free product of some copies of the group $\mathbb{Z}_{2}$ (these homomorphisms lead to evident composite homomorphisms from $\widetilde{G}_{n}^{3}$ to the free product mentioned above). In Secton 4, we get a sufficient condition for two adjacent generators $b_{ij}$ and $b_{ij}^{-1}$ not to cancel. In the present paper, we do not pretend to consturct necessary and suficient condition for cancellability; we just give some striking examles when this cancellability is impossible.

In the end of the paper, we give a list of unsolved problems and topics for future discussion.

\section{Homomorphisms from classical braids to ${\widetilde G_{n}^{3}}$}

Denote ${\bar n} :=\{1,\cdots, n\}$.

\begin{dfn} \cite{CohenFalkRandell}
The pure braids group $PB_{n}$ is the group given by group presentation generated by $\{b_{ij} ~|~ 1 \leq i<j \leq n\}$ subject to the following relations:
\begin{equation*}
b_{rs}b_{ij}b_{rs}^{-1} = \left\{
\begin{array}{cc} 
    b_{ij}, & \text{if}~ s<i  ~\text{or}~ j<r, \\
       b_{is}^{-1}b_{ij}b_{is}, & \text{if}~ i<j=r<s, \\
       b_{ij}^{-1}b_{ir}^{-1}b_{ij}b_{ir}b_{ij}, & \text{if}~ i<j<r=s, \\
       b_{is}^{-1}b_{ir}^{-1}b_{is}b_{ir}b_{ij}b_{ir}^{-1}b_{is}^{-1}b_{ir}b_{is},& \text{if}~ i<j<r<s. 
      
   \end{array}\right.
   \end{equation*}
\end{dfn}

\begin{dfn}\label{def_Gn3}
The group $G_{n}^{3}$ is the group given by group presentation generated by $\{ a_{\{ijk\}}~|~ \{i,j,k\} \subset \bar{n}, |\{i,j,k\}| = 3\}$ subject to the following relations:

\begin{enumerate}
\item $a_{\{ijk\}}^{2} = 1$ for $\{i,j,k\} \subset \bar{n}$, 
\item $a_{\{ijk\}}a_{\{stu\}} = a_{\{stu\}}a_{\{ijk\}}$, for $| \{i,j,k\} \cap \{s,t,u\} | < 2$,
\item $a_{\{ijk\}}a_{\{ijl\}}a_{\{ikl\}}a_{\{jkl\}} = a_{\{jkl\}}a_{\{ikl\}}a_{\{ijl\}}a_{\{ijk\}}$ for dictinct $i,j,k,l$.
\end{enumerate}
We denote $ a_{ijk} := a_{\{ijk\}}$.
\end{dfn}

Note that $a_{ijk} = a_{jik} = \cdots$, but $b_{ij} \neq b_{ji}.$
In \cite{ManturovNikonov} V.O.Manturov and I.N.Nikonov worked on that braids on $n$ strands can be considered as dynamical systems, in which $n$ distinct points move on the plane. Roughly speaking, the image of a braid in the group $G_{n}^{3}$ can be obtained by reading every moment, in which three points $(i,j,k)$ are placed on the same line. More precisely, the homomorphism $\phi_{n}$ from $PB_{n}$ to $G_{n}^{3}$ is defined by 
\begin{equation}
\phi_{n}(b_{ij}) = (c_{i,i+1})^{-1}(c_{i,i+2})^{-1} \cdots (c_{i,j-1})^{-1} (c_{i,j})^{2} c_{i,j-1} \cdots c_{i,i+2} c_{i,i+1},
\end{equation}
for each generator $b_{ij}$ of the group $PB_{n}$, $i, j \in \bar{n}, i<j$, where
\begin{equation}
c_{i,k} =  \prod_{l=k+1}^{n} a_{ikl} \prod_{l=1}^{k-1} a_{ikl},
\end{equation}
for $k \in \{i+1, \cdots, j\}.$

\begin{dfn} \label{def_tildeGn3}
 The group $\widetilde{G}_{n}^{3}$ is given by group presentation generated by $\{ a_{\{ijk\}}~|~ \{i,j,k\} \subset \bar{n}, |\{i,j,k\}| = 3\}$ and $\{ \sigma_{ij}~|~  i,j \in \{1, \dots, n\}, |\{i,j\}| = 2 \}$ subject to the following relations:
 
\begin{enumerate}\label{relations}
\item[(a)] $a_{\{ijk\}}^{2} = 1$ for $\{i,j,k\} \subset \{1, \cdots,n\}$, $|\{i,j,k\}| =3$, 
\item[(b)] $a_{\{ijk\}}a_{\{stu\}} = a_{\{stu\}}a_{\{ijk\}}$, if $| \{i,j,k\} \cap \{s,t,u\} | < 2$,
\item[(c)] $a_{\{ijk\}}a_{\{ijl\}}a_{\{ikl\}}a_{\{jkl\}} = a_{\{jkl\}}a_{\{ikl\}}a_{\{ijl\}}a_{\{ijk\}}$ for dictinct $i,j,k,l$,
\item[(d)] $\sigma_{ij}\sigma_{kl} = \sigma_{kl}\sigma_{ij}$ for dictinct $i,j,k,l$,
\item[(e)] $\sigma_{ij}a_{\{stu\}} = a_{\{stu\}}\sigma_{ij}$, if $| \{i,j\} \cap \{s,t,u\} | < 2$,
\item[(f)] $a_{\{ijk\}}\sigma_{ij}\sigma_{ik}\sigma_{jk} = \sigma_{jk}\sigma_{ik} \sigma_{ij}a_{\{ijk\}}$ for dictinct $i,j,k$,
\item[(g)] $\sigma_{ij}a_{\{ijk\}}\sigma_{ik}\sigma_{jk} = \sigma_{jk} \sigma_{ik} a_{\{ijk\}}\sigma_{ij}  $ for dictinct $i,j,k$,
\item[(h)] $\sigma_{ij}\sigma_{ik}a_{\{ijk\}}\sigma_{jk} = \sigma_{jk}a_{\{ijk\}} \sigma_{ik} \sigma_{ij} $ for dictinct $i,j,k$,
\item[(i)] $\sigma_{ij}\sigma_{ik}\sigma_{jk}a_{\{ijk\}} = a_{\{ijk\}}\sigma_{jk} \sigma_{ik}\sigma_{ij}  $ for dictinct $i,j,k$.

\end{enumerate}
We denote $a_{ijk} \sim a_{\{ijk\}}$; notice that $\sigma_{ij} \neq \sigma_{jk}$.
\end{dfn}

Following \cite{ManturovNikonov}, braids can be presented by dynamical systems, in which points move. But we consider braids on $n$ strands as $n$ moving points with one additional fixed (infinite) point. Let us define a mapping from $PB_{n}$ to $\widetilde{G}_{n}^{3}$. Now we consider pure braids as moving $n$ points on upper semi-disk. As the above, mapping from $PB_{n}$ to $\widetilde{G}_{n}^{3}$ will be defined by ``reading'' moments when some three points are on the same line, which is analogous to the construction of mapping from $PB_{n}$ to $G_{n}^{3}$. Let $n$ enumerated points $P = \{p_{1}, \cdots, p_{n}\}$ be placed on semicircle $\{ z \in \mathbb{C} ~|~ |z|=1, Imz\geq 0 \}$ in numerated order with respect to the courter-clockwise orientation. Let us place one more (infinite) point $p_{\infty}$ in the center of semicircle. When the points move, if three points $\{p_{i},p_{j},p_{k}\} \subset P$ are on the same line, we write the generator $a_{ijk}$. If points $p_{j},p_{i}, p_{\infty}$ are on the same directed line from $p_{\infty}$ in this order and the point $p_{i}$ passes the directed line from left to right, then we write the generator $\sigma_{ij}$, see Fig.~\ref{generators}.  For $i,j \in \bar{n}, i<j$ define
\begin{center}

\begin{eqnarray}\label{gen-phi}
c_{i,j} =(\prod_{k = j+1}^{n} a_{ijk}) \sigma_{ij}^{-1}  (\prod_{k =1}^{j-1} a_{ijk}), \\
\bar{c_{i,j}} =(\prod_{k = j+1}^{n} a_{ijk}) \sigma_{ij}  (\prod_{k =1}^{j-1} a_{ijk}),\\
c_{j,i} =(\prod_{k = j+1}^{n} a_{ijk}) \sigma_{ji}^{-1}  (\prod_{k =1}^{j-1} a_{ijk}),\\
\bar{c_{j,i}} =(\prod_{k = j+1}^{n} a_{ijk}) \sigma_{ji}  (\prod_{k =1}^{j-1} a_{ijk}). 
\end{eqnarray}

\end{center}
Then the mapping $\widetilde{\phi}$ from $PB_{n}$ to $\widetilde{G}_{n}^{3}$ is defined by 

 \begin{equation}\label{def-phi}
\widetilde{\phi}(b_{ij}) = c_{i,i+1}^{-1} c_{i,i+2}^{-1}\cdots c_{i,j-1}^{-1} \bar{c_{i,j}} \bar{c_{j,i}} c_{i,j-1}\cdots c_{i,i+2}c_{i,i+1}.
 \end{equation}

\begin{figure}[h]
\begin{center}
 \includegraphics[width = 15cm]{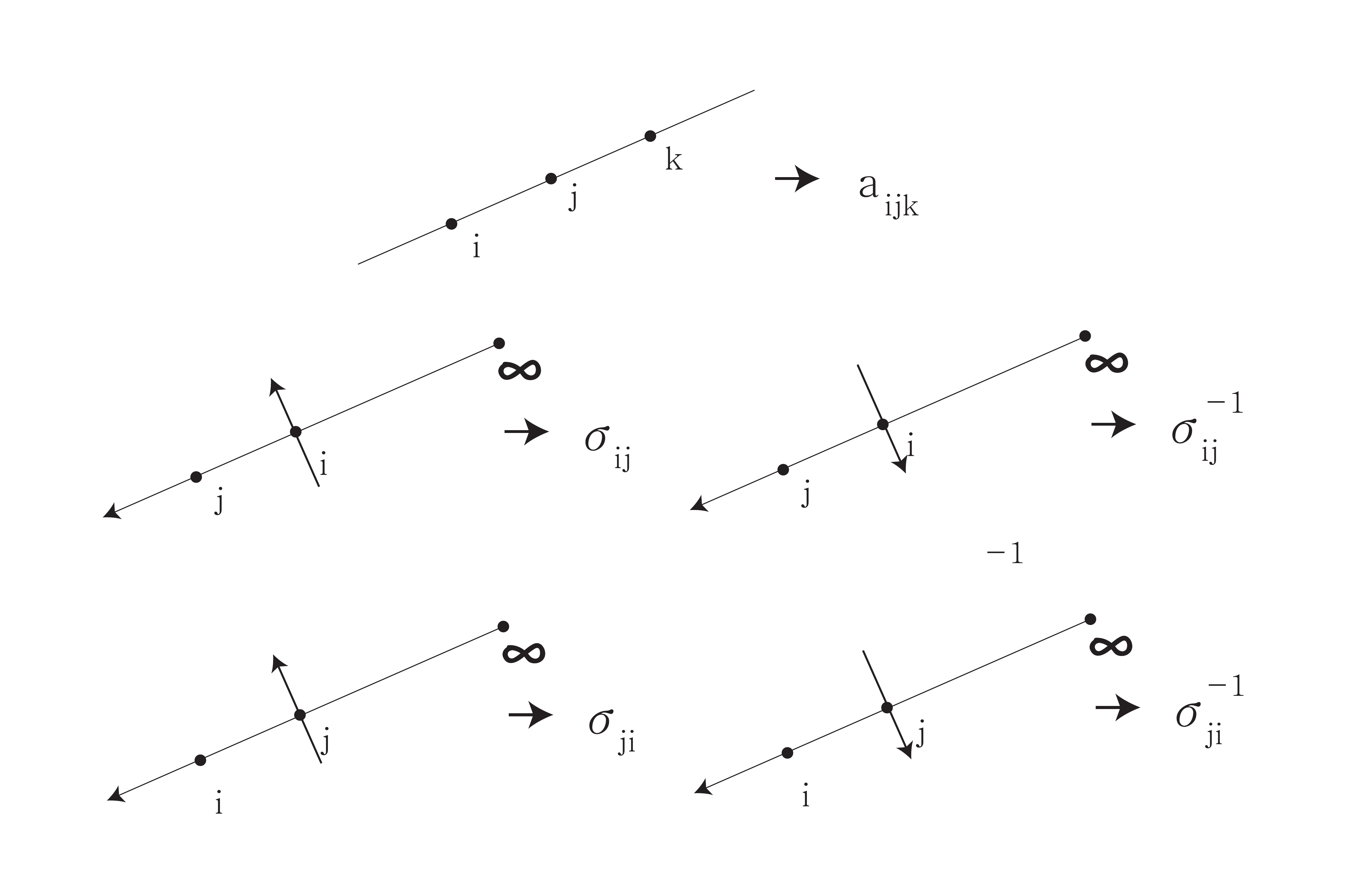}

\end{center}

\caption{Generators related for each cases}\label{generators}
\end{figure}

\begin{exa}
Let $n=6$, $i=2, j=4$ and points are placed on semi-circle shown in Fig.~\ref{exa1}. By reading every triple points, we obtain the following word in the group $\widetilde{G}_{6}^{3}$:
$$\widetilde{\phi}(b_{24}) = a_{123}\sigma_{23}a_{236}a_{235}a_{234} a_{234}a_{124}\sigma_{24}a_{246}a_{245}   a_{234}a_{124}\sigma_{42}a_{246}a_{245}  a_{234}a_{235}a_{236}\sigma_{23}^{-1}a_{123},$$
see Fig.~\ref{exa1-2}.
By definition, we obtain
 \begin{equation}
 \widetilde{\phi}(b_{24}) = c_{23}^{-1}\bar{c_{24}}\bar{c_{42}}c_{23}.
 \end{equation}
 
\begin{figure}[h]
\begin{center}
 \includegraphics[width = 12cm]{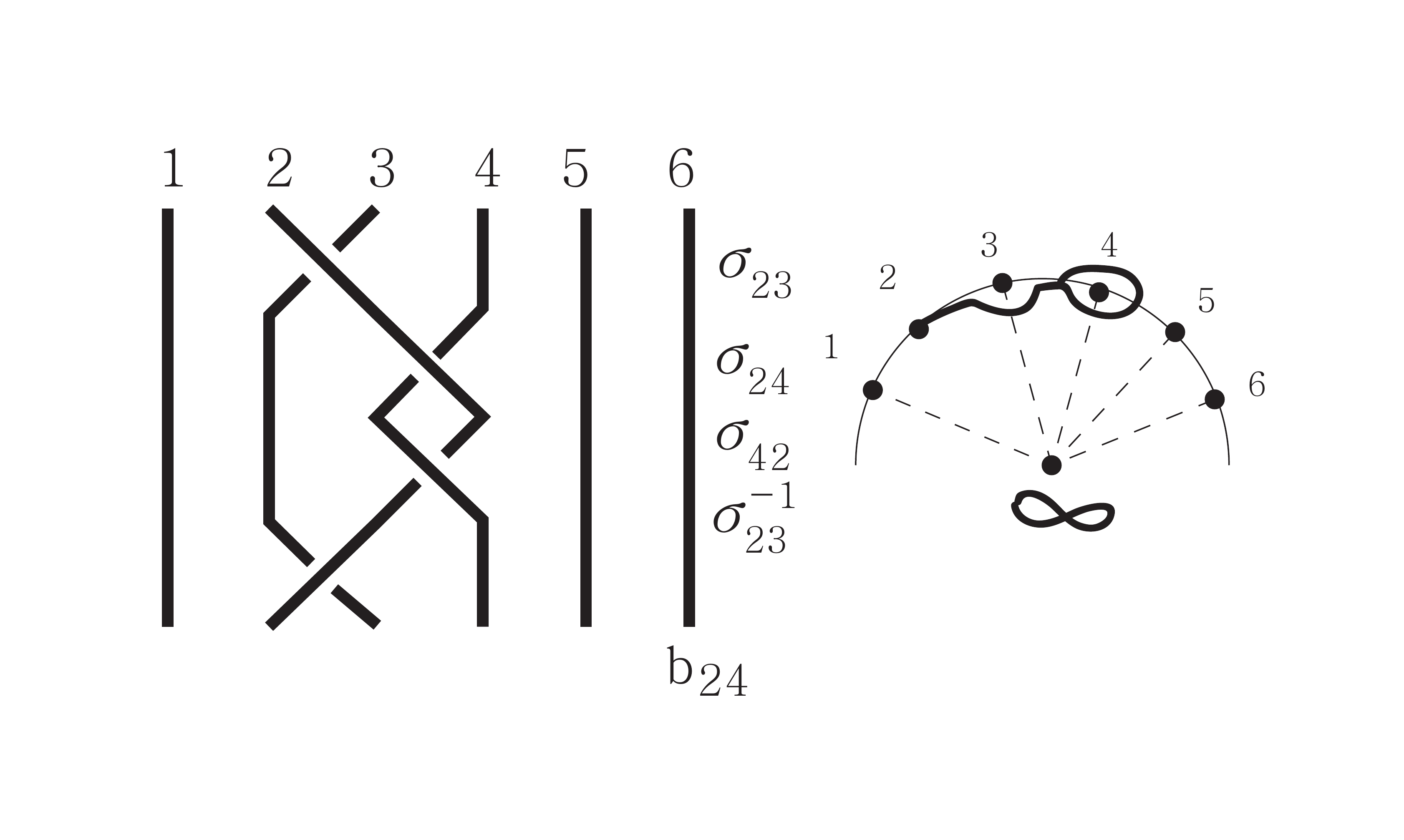}

\end{center}

\caption{A generator $b_{24}$ in $PB_{6}$ and its dynamical system}\label{exa1}
\end{figure}

\begin{figure}[h]
\begin{center}
 \includegraphics[width = 15cm]{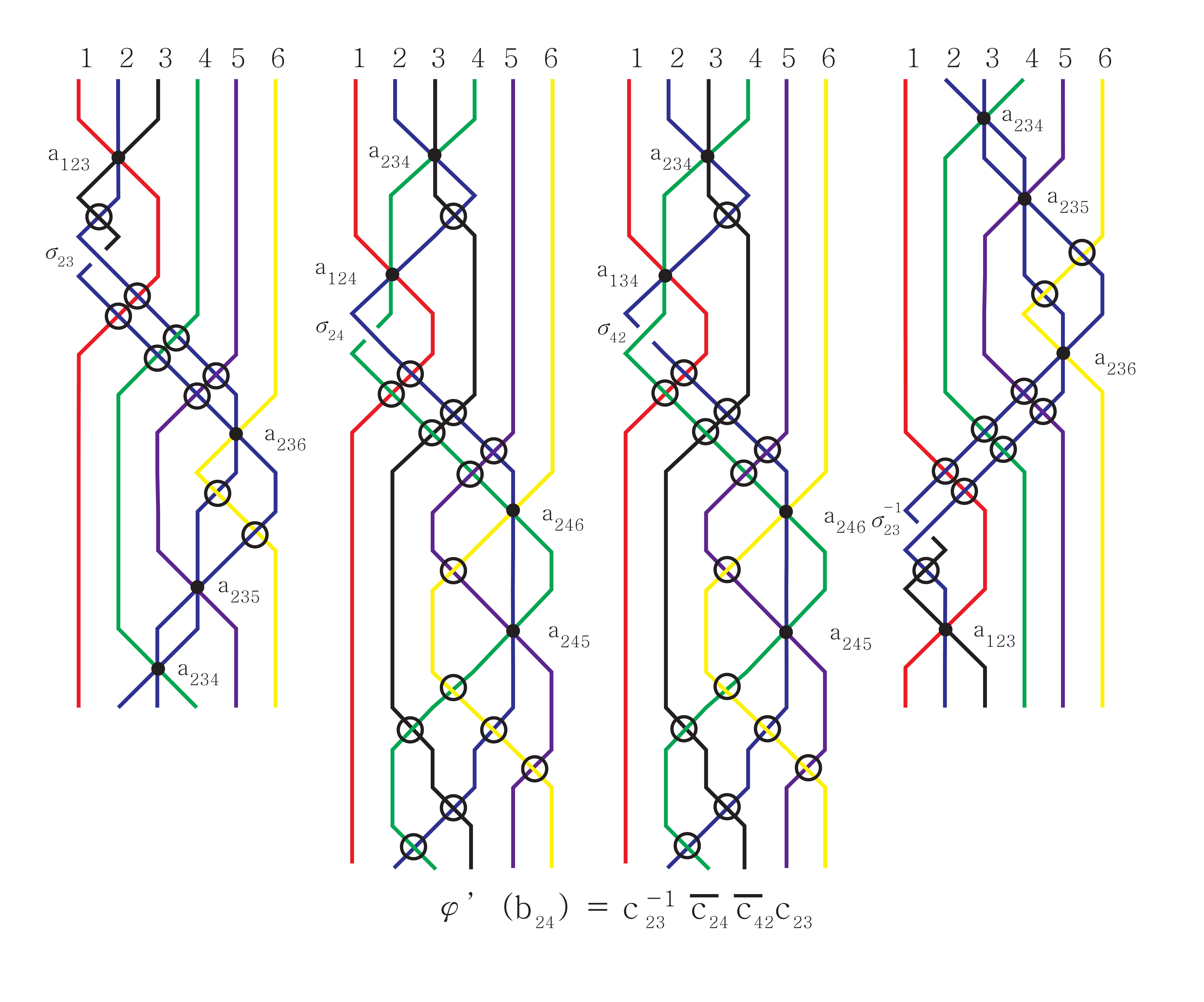}

\end{center}

\caption{The diagram of the image of $b_{24}$ along $\widetilde{\phi}$}\label{exa1-2}
\end{figure}

\end{exa}

The sets of codimension 1 is related to generators, and the sets of codimension 2 are related to relations. They are analogous to those from~\cite{ManturovNikonov} in the case of (n+1) points. In this work, if a point passes to triple point, which contains the infinite point ($\infty$), then it is possible to give an additional information and write one of generators not of $a_{ij\infty}$, but of $\sigma_{ij},\sigma_{ij}^{-1},\sigma_{ji}$ or $\sigma_{ji}^{-1}$. This is the main idea to prove the theorem given below, but this theorem differs from the main theorem in \cite{ManturovNikonov}, roughly speaking, the image has more generators and relations than the image in \cite{ManturovNikonov}.

\begin{thm}\label{homo_tildeGn3}
The mapping $\widetilde{\phi}$ from $PB_{n}$ to $\widetilde{G}_{n}^{3}$, defined by (\ref{def-phi}), is a homomorphism. In other words, if pure braids $\beta,\beta'$ are equivalent in $PB_{n}$, then $\widetilde{\phi}(\beta)=\widetilde{\phi}(\beta')$ in $\widetilde{G}_{n}^{3}$.
\end{thm}

\begin{proof}

Let $z_{k} = e^{\pi i (n-k) / (n-1)},$ $k=1, \cdots ,n,$ be points on semicircle $C = \{z = a+bi \in \mathbb{C} ~|~ |z| =1, b\geq 0 \}$. Pure braids can be considered as dynamical systems, in which points move on the plane such that the position of points in the start and in the end are same as asserted in the above. 
Now we clearly formulate the image of each generator in $PB_{n}$. For $i<j$ a generator $b_{ij}$ can be considers as the following dynamical system (see Fig.~\ref{c(ij)1}):
\begin{enumerate}
\item The point $z_{i}$ moves along the semi-circle $C$ passes beside $z_{i+1},z_{i+2}, \cdots, z_{j-1}$ to the point $z_{j}$. 
\item The point $z_{i}$ turn around $z_{j}$ in the counter-clockwise orientation. 
\item The point $z_{i}$ comes back to the initial position beside $z_{j-1}, \cdots, z_{i+1}$.
\end{enumerate}

\begin{figure}[h]

 \includegraphics[width =17cm]{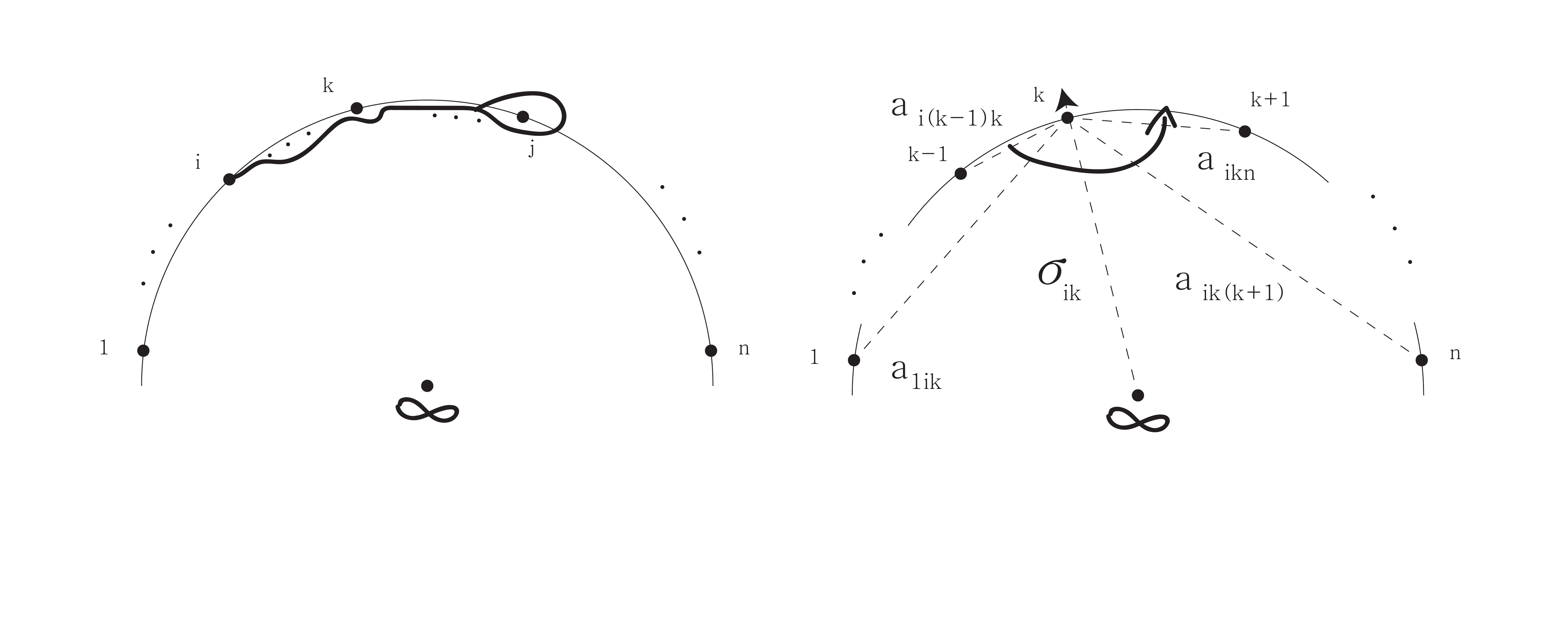}

\caption{The dynamical systems associating with $b_{ij}$ and $c_{ik}$}\label{c(ij)1}
\end{figure}

\begin{figure}[h]
\begin{center}
 \includegraphics[width = 8cm]{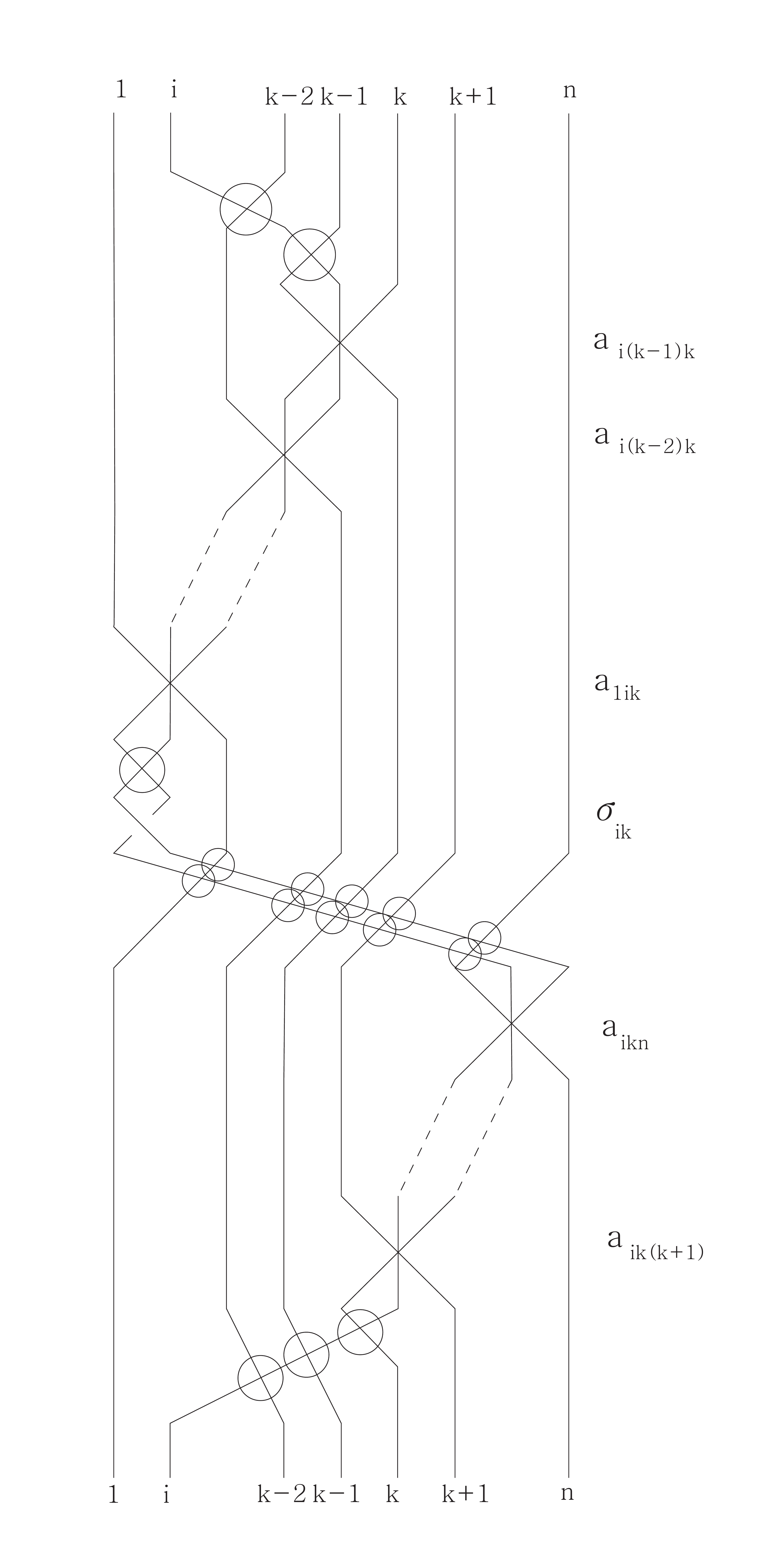}

\end{center}

\caption{The diagram of $c_{ij}^{-1}$}\label{c(ij)2}
\end{figure}
When three points are placed on the same straight line, we write one of generators of $\widetilde{G}_{n}^{3}$ with respect to the rules, which are shown in Fig.~\ref{generators}. In the end of the process, we obtain a word
$$\widetilde{\phi}(b_{ij}) = c_{i,i+1}^{-1} c_{i,i+2}^{-1}\cdots c_{i,j-1}^{-1} \bar{c_{i,j}} \bar{c_{j,i}} c_{i,j-1}\cdots c_{i,i+2}c_{i,i+1}.$$
Now we will show that if two pure braids $\beta$ and $\beta'$ are equivalent in $PB_{n}$, then $\widetilde{\phi}(\beta)$ and $\widetilde{\phi}(\beta')$ are equivalent in $\widetilde{G}_{n}^{3}$.
Firstly, we introduce the following two definitions:
a dynamical system is {\it nice,} if the following holds: 
\begin{description}
\item[P1] If points are placed on the same straight line, then the number of the points is less than or equal to $3$.
\item[P2] For each moment there is at most one triple of points $\{p_{i}, p_{j}, p_{k} \}$, which are collinear. If such a triple of point happens, we call it the {\it critical moment of type $\{i,j,k\}$.}
\item[P3] The number of critical moments is finite.
\end{description}

A dynamical system $D$ is called {\it stable}, if it satisfies the followings:

\begin{description}

\item[C1] Every dynamical system $D'$ in the neighborhood $U(D)$ of $D$ (that is, $D'$ is obtained from $D$ by transforming $D$) is nice,
\item[C2] For each dynamical system $D'$ in $U(D)$, $\widetilde{\phi}(D) = \widetilde{\phi}(D')$.

\end{description}

Without loss of generality we may assume that the pure braids have the forms of nice stable dynamical system. 

Let $\{ \beta_{t} \}_{t \in I}$ be a isotopy such that $\beta_{0} = \beta$ and $\beta_{1} = \beta'$. Without loss of generality, we may assume that $\{ \beta_{t} \}_{t \in I}$ satisfies the followings:
\begin{enumerate}
\item For every $t \in (t_{0}, t_{1}) \subset I$, if $\beta_{t}$ is stable and nice, then the set of critical moments of $\beta_{t}$ changes continuously.
\item In $I$ there are finite $s_{0} \in I$ such that $\beta_{s_{0}}$ is not nice or not stable. In other words, $\beta_{s_{0}}$ satisfies one of the followings;
\begin{description}
\item[A] Suppose that there are three points $\{z_{i},z_{j},z_{k}\} \subset \{z_{1} \cdots z_{n}\}$, such that the points $\{z_{i},z_{j},z_{k}\}$ are on the same line in the moment $t$ and they are still on the same line in the first approximation centered at $t$. 
For some $\epsilon > 0$, the word $\widetilde{\phi}(\beta_{s_{0}-\epsilon})$ has the form of $Fa_{ijk}a_{ijk} B$, but $\widetilde{\phi}(\beta_{s_{0}+\epsilon})$ has the form of $FB$ (see. Fig.~\ref{proof_mixed0}). That is, when $\beta_{t}$ passes the moment $\beta_{s_{0}}$, $\widetilde{\phi}(\beta_{s_{0}+\epsilon})$ is obtained from $\widetilde{\phi}(\beta_{s_{0}-\epsilon})$ by the relation (a) in Definition~\ref{def_tildeGn3}.

\item[B] Suppose that four points are on the same line in the moment $\beta_{s_{0}}$ (see. Fig.~\ref{proof_mixed1}). If they have not the infinite point, then for some $\epsilon > 0$, $\widetilde{\phi}(\beta_{s_{0}-\epsilon})$ contains a product of $a_{ijk},a_{ijl}, a_{ikl}, a_{jkl}$ in some order, $\widetilde{\phi}(\beta_{s_{0}+\epsilon})$ have the product of $a_{ijk},a_{ijl}, a_{ikl}, a_{jkl}$ in the reverse order.  When $t$ is changed, $\widetilde{\phi}(\beta_{s_{0}+\epsilon})$ is obtained from $\widetilde{\phi}(\beta_{s_{0}-\epsilon})$ by the relation (c) in Definition~\ref{def_tildeGn3}.

If one of the four points is the infinite point, then for some $\epsilon > 0$, the word $\widetilde{\phi}(\beta_{s_{0}-\epsilon})$ has a product of $a_{ijk},\sigma_{ij},\sigma_{ik},\sigma_{jk}$ in some order, and the word $\widetilde{\phi}(\beta_{s_{0}+\epsilon})$ has the product of  $a_{ijk},\sigma_{ij},\sigma_{ik},\sigma_{jk}$ in the reverse order. When $\beta_{t}$ passes the moment $\beta_{s_{0}}$, the word $\widetilde{\phi}(\beta_{s_{0}+\epsilon})$ is obtained from the word $\widetilde{\phi}(\beta_{s_{0}-\epsilon})$ by the relations (f),(g),(h),(i) of relations of the group  $\widetilde{G}_{n}^{3}$ in Definition~\ref{def_tildeGn3}.

\item[C] Suppose that two sets $m$ and $m'$ of three points on the lines $l$ and $l'$ respectively in the moment $s_{0}$, such that $ | m \cap m' | < 2 $ (see. Fig.~\ref{proof_mixed2}). Then $\widetilde{\phi}(\beta_{t})$ is changed according to one of the relations (b),(d),(e) of relations of the group  $\widetilde{G}_{n}^{3}$ in Definition~\ref{def_tildeGn3}. 

\end{description}
\end{enumerate}

\begin{figure}[h]
\begin{center}
 \includegraphics[width = 12cm]{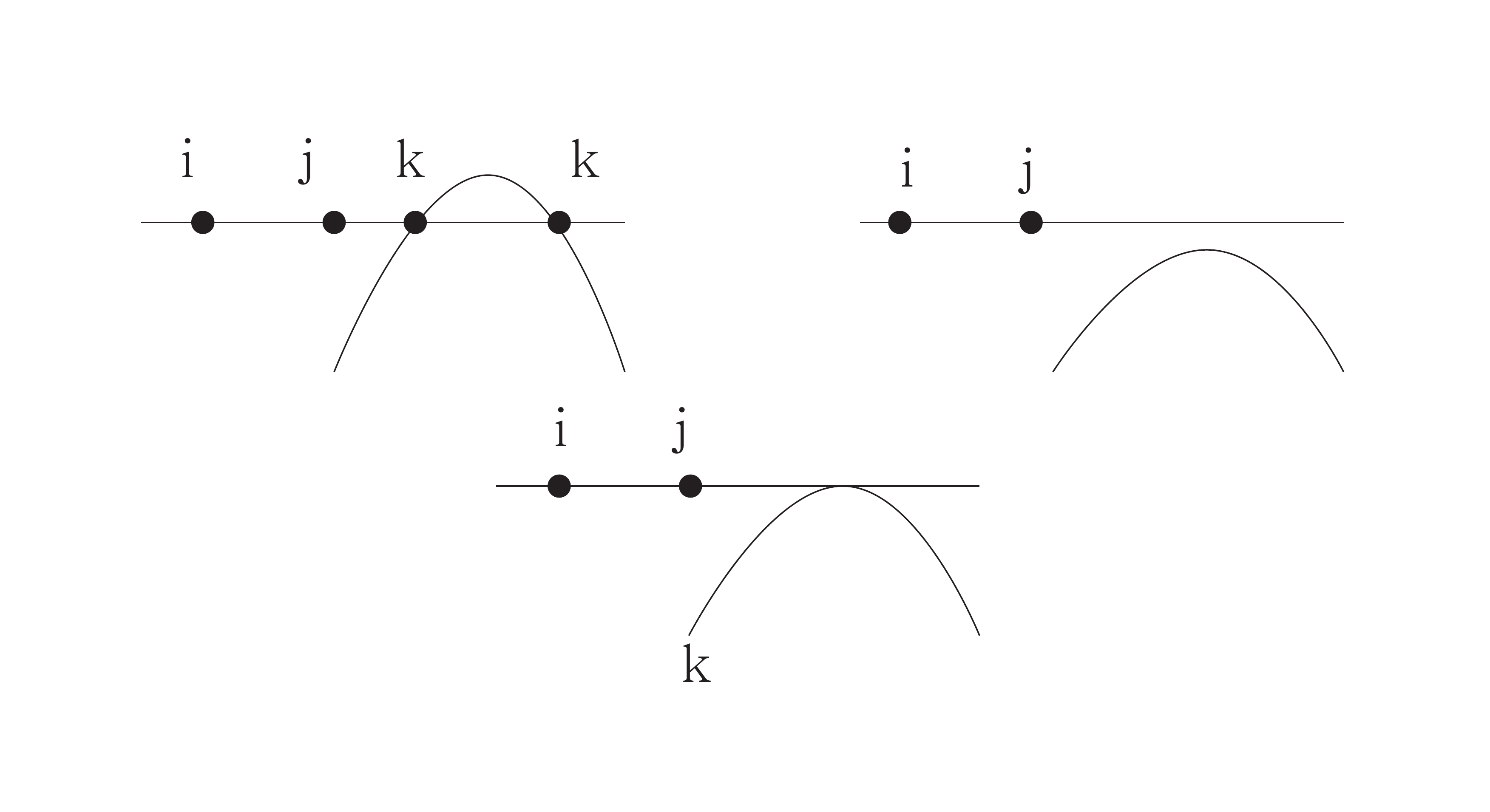}

\end{center}
\caption{Case A}\label{proof_mixed0}
\end{figure}

\begin{figure}[h]

 \includegraphics[width = 10cm]{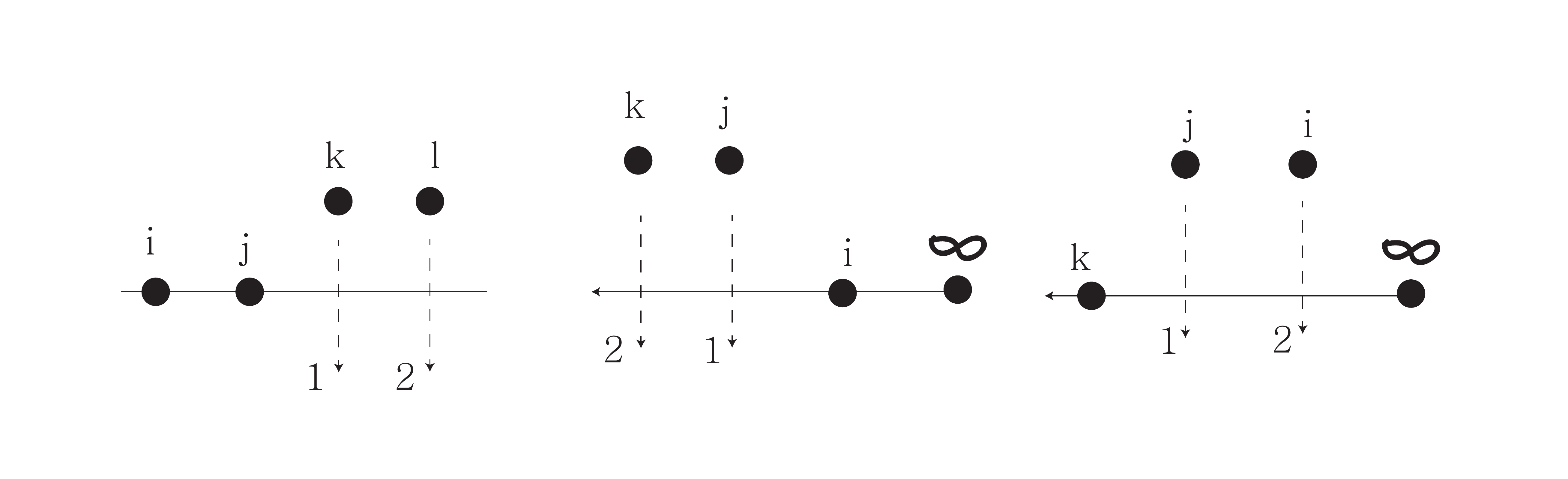}

\caption{Case B}\label{proof_mixed1}
\end{figure}

\begin{figure}[h]
\begin{center}
 \includegraphics[width = 10cm]{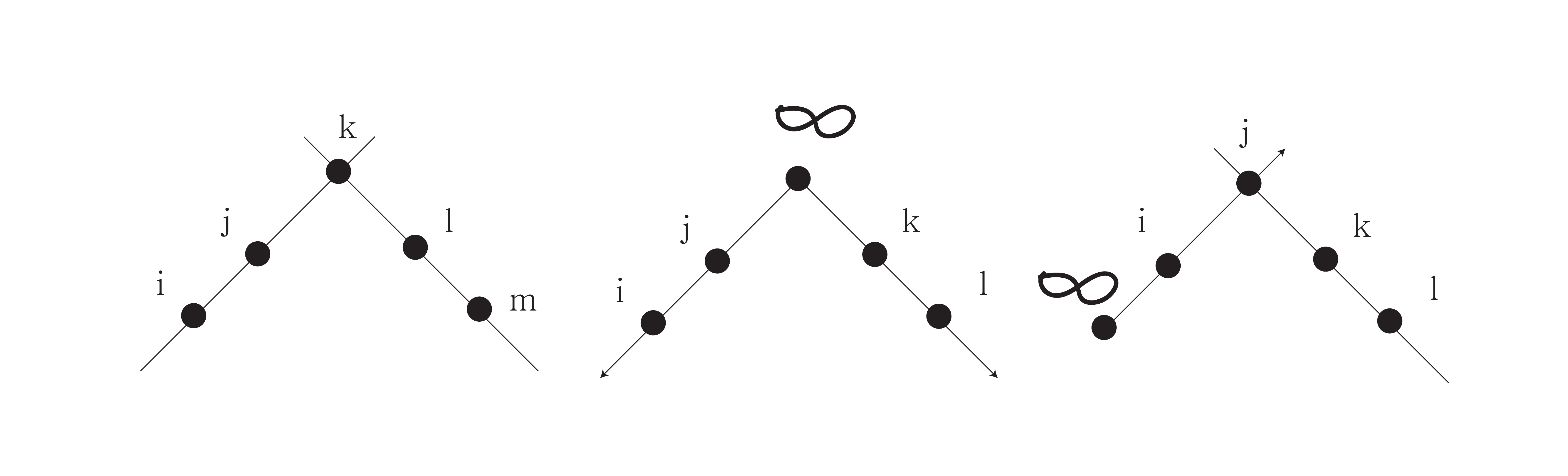}

\end{center}

\caption{Case C}\label{proof_mixed2}
\end{figure}

As the above, we can rewrite every type of deformations (codimension 2), which correspond to general position isotopies between two braids. Passing those moments which are either not good or not stable, the word is deformed by one of relations of the group $\widetilde{G}_{n}^{3}$ and the proof is completed.

\end{proof}

\section{Homomorphisms from $\widetilde{G}_{n}^{3}$ to $G_{n+1}^{3}$}

Define the homomorphism $pr : \widetilde{G}_{n}^{3} \rightarrow G_{n}^{3}$ by

\begin{equation}
pr(a_{ijk})  = \left\{
\begin{array}{cc} 
    1, & \text{if}~ \infty \in \{i,j,k\}, \\
       a_{ijk}, & \text{if}~ \infty \not\in \{i,j,k\}.
      
   \end{array}\right.
   \end{equation}
The proposition below follows from the definition of $pr$:
 \begin{prop}
  $pr \circ \widetilde{\phi} = \phi$.
 \end{prop}
Now we define homomorphism $i$ from $G_{n}^{3}$ to $\widetilde{G}_{n}^{3}$ by $i(a_{ijk}) = a_{ijk}$. Then $i \circ pr = Id_{G_{n}^{3}}$.

Besides, we define the homorphism $\pi$ from $\widetilde{G_{n}^{3}}$ to $G_{n+1}^{3}$ by $\pi(a_{ijk}) =a_{ijk}$ and $\pi(\sigma_{ij}) = a_{ij(n+1)}$. 
In~\cite{ManturovNikonov} V.O.Manturov and I.M.Nikonov studied homomorphisms from $G_{n}^{k}$ to the free product of copies of groups $\mathbb{Z}_{2}$. This, in turn, leads to the homomorphism $w_{ijk} : G_{n}^{3} \rightarrow F_{n}^{3}$ is constructed, where
$$F_{n}^{3} = \langle \{ \sigma ~|~ \sigma : \{1,2,\cdots n\} \backslash \{i,j,k\} \rightarrow \mathbb{Z}_{2} \times \mathbb{Z}_{2} \} ~|~ \{ \sigma^{2} = 1\} \rangle \cong \mathbb{Z}_{2}^{*2^{2(n-3)}}.$$
For each generator $a_{ijk}$ in $\beta = F a_{ijk} B \in G_{n}^{k}$, let us define the mapping $i_{a_{ijk}} : \bar{n} \backslash \{ i,j,k \} \rightarrow \mathbb{Z}_{2} \times \mathbb{Z}_{2}$ by
$$i_{c}(l) = (N_{jkl}+N_{ijl}, N_{ikl}+N_{ijl}) \in \mathbb{Z}_{2} \times \mathbb{Z}_{2},$$
for $l \in  \bar{n} \backslash \{ i,j,k \}$, where $N_{jkl}$ is the number of occurencies of $a_{ikl}$ in $F$. We call $i_{a_{ijk}}$ ``index'' of $a_{ijk}$ in $\beta$. Let $\{c_{1}, \cdots, c_{m}\}$ be the ordered set of all $a_{ijk}$ in $\beta$ such that if $\beta = T_{l} c_{l} B_{l}$, then $c_{s} \in T_{l}$ for $s<l$. Define $w_{ijk} : G_{n}^{3} \rightarrow F_{n}^{3}$ by
$$w_{(i,j,k)}(\beta) = i_{c_{1}}i_{c_{2}} \cdots i_{c_{m}}.$$

Note that if generators on the right hand side have indices, then we can define indices for generators on the left hand side. In other words, if we get a homomorphism $x : G \rightarrow H$ for generators of the group $G$ and $H$, and if indices $i$ for generators in the group $H$ are defined,
$$ \cdots a \cdots \rightarrow \cdots x(a) \cdots $$
then the indices $j$ for generators in $G$ can be defined by
$$ j(a) := i(x(a)).$$
That is, for a generator $b$ in $\widetilde{G_{n}^{3}}$, the index $j_{b}$ can be defined by means of $G_{n+1}^{3}$ as follow:
$$j_{b}(l) := i_{\pi(b)}(l),$$
for each $l \in  \overline{n+1} \backslash \{ i,j,k \}$ .
Analogously we define the homomorphism from $\widetilde{G_{n}^{3}}$ to the free product of copies of groups $\mathbb{Z}_{2}$ as follow:
$$\widetilde{w}_{ijk} := w_{(i,j,k)} \circ \pi : \widetilde{G_{n}^{3}} \rightarrow G_{n+1}^{3} \rightarrow {Z}_{2}^{*2^{2((n+1)-3)}}.$$

\begin{exa}\label{index-classical} For a braid $b_{24} \in PB_{6}$,\\
$\pi \circ \widetilde{\phi}(b_{24}) =\pi(a_{123}\sigma_{23}a_{236}a_{235}a_{124}\underline{\sigma_{24}}a_{246}a_{245}a_{234}a_{124}\underline{\sigma_{42}}a_{246}a_{245}a_{234}a_{235}a_{236}\sigma_{23}^{-1}a_{123})$\\
$=a_{123}a_{237}a_{236}a_{235}a_{124}\underline{a_{247}}a_{246}a_{245}a_{234}a_{124}\underline{a_{247}}a_{246}a_{245}a_{234}a_{235}a_{236}a_{237}a_{123}.$\\
Then 
\begin{center}
$ i_{a_{247}}= \left (
\begin{array}{cccc} 
   1 & 1  & 0  & 0\\
   1  & 0  & 0 & 0\\
  
   \end{array} \right )$
and $ i_{a_{247}}= \left (
\begin{array}{cccc} 
   0 & 0  & 1  & 1\\
   0  & 1  & 1 & 1\\
  
   \end{array} \right ), $ 
   
   \end{center}
where $i_{c} = (i_{c}(1),i_{c}(3),i_{c}(5),i_{c}(6)).$ We define indices $j$ for generators $\sigma_{24}$ and $\sigma_{42}$ with respect to $i$.

\begin{center}
$j_{\sigma_{24}} := i_{a_{247}}= \left (
\begin{array}{cccc} 
   1 & 1  & 0  & 0\\
   1  & 0  & 0 & 0\\
  
   \end{array} \right )$
and $j_{\sigma_{42}} := i_{a_{247}}= \left (
\begin{array}{cccc} 
   0 & 0  & 1  & 1\\
   0  & 1  & 1 & 1\\
  
   \end{array} \right ). $ 
   
   \end{center}
\end{exa}

\section{Non-cancellable generators $b_{ij}$ and $b_{ij}^{-1}$}

We will use the group $\widetilde{G}_{n}^{3}$ to know whether two classical crossing cannot be canceled or not.
\begin{exa}\label{newindex-brunnian}
Let $\beta = [b_{12},b_{13}]$ in $PB_{3}$. 
\begin{eqnarray*}
\pi \circ \widetilde{\phi}(\beta) &=&\pi(\sigma_{12}a_{123}\sigma_{12}a_{123}\sigma_{12}\sigma_{13}a_{123}\sigma_{13}a_{123}\sigma_{12}^{-1}a_{123}\sigma_{12}^{-1}\sigma_{13}^{-1}a_{123}\sigma_{13}^{-1}\sigma_{12}^{-1})\\
&=&a_{124}a_{123}a_{124}a_{123}a_{124}a_{134}a_{123}a_{134}a_{123}a_{124}a_{123}a_{124}a_{134}a_{123}a_{134}a_{124}
\end{eqnarray*}
in $G_{4}^{3}$. Then $$\widetilde{w}_{123}(\widetilde{\phi}(\beta)) = (1,0)(0,0)(1,1)(1,0)(0,0)(1,1)  \neq 1$$
in $F_{4}^{3}$, where $i_{a_{123}}(4) = (N_{124} + N_{234},N_{134} + N_{234})$ mod $2$, and $\beta$ is non trivial braid in $PB_{3}$. Note that the braid $\beta$ is Brunnian, and $w_{ijk} \circ \phi (\beta) =1$ for Brunnian braids $\beta$, that is, $w_{123}(\phi(\beta)) =1$ (Theorem 3.4 from \cite{KimManturov}).
\end{exa}
\begin{thm}\label{obstacletocancle}
For pure braids in the form of $\beta = Ab_{ij}Bb_{ij}^{-1}C \in PB_{n}$, where $A,B,C \in PB_{n}$, if the total number of $b_{ik}$ and $b_{jk}$ in $B$ is odd, and there are no $b_{ij}$ and $b_{ij}^{-1}$ in $B$, then $b_{ij}$ and $b_{ij}^{-1}$ cannot be canceled by relations of $PB_{n}$.
\end{thm}

\begin{proof}
Without loss of generality, assume that $i<j<k$. Note that $\widetilde{w}_{ijk} \circ \widetilde{\phi}(\beta) = \widetilde{w}_{ijk} \circ \widetilde{\phi}(f_{ijk}(\beta))$, where $f : PB_{n} \rightarrow PB_{n}$ is the endomorphism defined by
\begin{equation}
f_{ijk}(b_{st})  = \left\{
\begin{array}{cc} 
    1, & \text{if}~ |\{s,t\}  \cap \{i,j,k\} | \neq 2, \\
      b_{st}, & \text{if}~ |\{s,t\}  \cap \{i,j,k\}| = 2.
      
   \end{array}\right.
   \end{equation}
The homomorphism $\widetilde{w}_{ijk}$ is valued on the free product of copies of $\mathbb{Z}_{2}$ and it follows that $\widetilde{w}_{ijk} \circ \widetilde{\phi}(f_{ijk}(B)) =1$ for two generators $b_{ij}$ and $b_{ij}^{-1}$ to be cancelled. By the assumption we have $f_{ijk}(B) = c_{1}\cdots c_{k}$ for $k \equiv 1$ mod $2$, where $c_{i} \in \{b_{ik},b_{ik}^{-1},b_{jk}, b_{jk}^{-1}\}$.
Now we calculate $\widetilde{w}_{ijk} \circ \widetilde{\phi}(b_{ij}) = w_{ijk} \circ \pi \circ  \widetilde{\phi}(b_{ij})$, where $i_{a_{ijk}}(n+1) = (a_{ij(n+1)}+a_{jk(n+1)}, a_{ik(n+1)}+a_{jk(n+1)}) \in \mathbb{Z}_{2} \times \mathbb{Z}_{2}$. By the definition it follows that

\begin{eqnarray*}
\pi \circ \widetilde{\phi}(b_{ij}) &=& \pi(c_{i,i+1}^{-1} \cdots c_{i,j-1}^{-1} \bar{c_{i,j}}\bar{c_{j,i}}c_{i,j-1} \cdots c_{i,j-1}) \\
&=& c_{i,i+1}^{-1}\cdots c_{i,j-1}^{-1} c_{i,j}c_{i,j}c_{i,j-1}\cdots c_{i,j-1}.
\end{eqnarray*}

In the words $c_{i,j}$,$c_{i,k}$,$c_{j,k}$ there are generators $a_{ijk}$, $a_{ij(n+1)}$, $a_{ik(n+1)}$, $a_{jk(n+1)}$, but they are not in $c_{s,t}$ for $\{s,t\} \cap \{i,j,k\} < 2$. Therefore $c_{s,t}$ for $\{s,t\} \cap \{i,j,k\} < 2$ do not affect to $i_{a_{ijk}}(n+1)$, and we focus on $w_{ijk}(c_{i,j}c_{i,j})$.
Then
$$w_{ijk}(c_{i,j}c_{i,j}) = a_{ij(j+1)} \cdots a_{ijk} \cdots a_{ij(n+1)} a_{1ij} \cdots a_{ij(i-1)}a_{ij(j+1)} \cdots a_{ijk} \cdots a_{ij(n+1)} a_{1ij} \cdots a_{ij(i-1)},$$
and there are exactly two $a_{ijk}$, denote them by $c_{1}$ and $c_{2}$. From simple calculations, we obtain that $w_{ijk} \circ \pi \circ \widetilde{\phi}(b_{ij}) = i_{c_{1}}i_{c_{2}}$, where  
$$i_{c_{1}}(n+1) = (0,0), i_{c_{2}}(n+1) = (1,0).$$
 Analogously, we can get that $\widetilde{w}_{ijk} \circ \widetilde{\phi}(b_{ik}) = i_{c_{1}}i_{c_{2}}i_{c_{3}}i_{c_{4}}$, where $$i_{c_{1}}(n+1) = (1,0), i_{c_{2}}(n+1) = (1,1), i_{c_{3}}(n+1) = (1,0), i_{c_{4}}(n+1) = (1,0),$$
and $\widetilde{w}_{ijk} \circ \widetilde{\phi}(b_{jk}) = i_{c_{1}}i_{c_{2}}$, where 
  $$i_{c_{1}}(n+1) = (1,1), i_{c_{2}}(n+1) = (0,0).$$ 
  That is, the images of the homomorphism $w_{ijk} \circ \pi \circ \widetilde{\phi}$ of $\{b_{ij},b_{ik},b_{jk} \}$ have different values. It is easy to show that if $w_{ijk} \circ \pi \circ \widetilde{\phi}(P)=1$ for a product $P$ of $\{b_{ij},b_{ik},b_{jk} \}$, then $P$ is the product of words $(b_{ij}b_{ik}b_{jk})^{\pm 1}$ and $(b_{jk}b_{ik}b_{ij})^{\pm 1}$ up to relations in $PB_{n}$. 
 But $k \equiv 1$ mod $2$, and there are no $b_{ij}$ in $B$,
$$\widetilde{w}_{ijk} \circ \widetilde{\phi}(f_{ijk}(B)) =  \widetilde{w}_{ijk} \circ \widetilde{\phi}(c_{1}\cdots c_{k}) \neq 1,$$
and hence $b_{ij}$ and $b_{ij}^{-1}$ cannot be cancelled.
\end{proof}

\begin{rem}
Let $\beta$ be a braid on $n$ strands and let two strands $i$ and $j$ be fixed. Let $c_{1},c_{2}$ be classical crossings $\sigma_{ij},\sigma_{ij}^{-1}$ between $i$ and $j$ strands. 
Is it possible for $c_{1}$ and $c_{2}$ in $\beta$ to be cancelled? To answer the question, we need to use indices for crossings $c_{1},c_{2}$. 

As asserted in the previous section, the indices are defined by the homomorphism from $PB_{n}$ to $\widetilde {G}_{n}^{3}$. That is, to obtain the ``local'' information (indices for crossings) we use ``the global'' information (homomorphisms $\widetilde{w}_{ijk}$). As Theorem~\ref{obstacletocancle} indices (obstruction to reducing) work in groups $G_{n}^{k}$. 

In the simplest case, $G_{n}^{2}$, ``an algorithm of descending'' takes place, in other words, if a word $\beta$, which is obtained from $G_{n}^{2}$ and is not minimal, then there must be two ``cancellable'' crossings, i.e. $\beta= X a_{ij} Y a_{ij} Z$, and $\beta$ can be converted into $\beta' = X {\tilde Y} a_{ij} a_{ij} Z$ by relations of $G_{n}^{2}$, without change of the length of the word, and in the final word $\beta'$ there are two crossings $a_{ij}$, which are directly reduced. Step by step, we obtain a word of minimal length. Moreover, if two words $\beta'$ and $\beta''$ obtained from a word $\beta$ have the minimal length, then there is a sequence $S$ of relations between $\beta'$ and $\beta''$ such that every word, which is obtained by a subsequence of $S$ from $\beta'$, does not have reducing crossings. It means that Diamond lemma holds for the standard group presentation of $G_{n}^{2}$. In~\cite{Manturov2} this is proved by means of the Coxeter groups.

In the case of groups $G_{n}^{3}$ the Diamond lemma does not take place: there is an element, which has two different representatives of the minimal length in the group $G_{4}^{3}$. For example, let words
$$a_{123}a_{124}a_{134}a_{234}a_{234}a_{124}a_{134}a_{123} ~\text{and}~ a_{234}a_{134}a_{124}a_{123}a_{123}a_{134}a_{123}a_{234}$$ 
in the group $G_{4}^{3}$. They are equivalent by the following sequence of deformations
\begin{eqnarray*}
a_{123}a_{124}a_{134}a_{124}a_{134}a_{123}  &=& a_{123}a_{124}a_{134}a_{234}a_{234}a_{124}a_{134}a_{123} \\
&=& a_{234}a_{134}a_{124}a_{123}a_{123}a_{134}a_{123}a_{234} \\
&=& a_{234}a_{134}a_{124}a_{134}a_{123}a_{234}.
\end{eqnarray*}
To the words $a_{123}a_{124}a_{134}a_{124}a_{134}a_{123}$ and $a_{234}a_{134}a_{124}a_{134}a_{123}a_{234}$ any relations from Definition~\ref{def_Gn3} cannot be applied except for relations $a_{ijk}^{2}=1$, that is, they have minimal lengths. But, without relations $a_{ijk}^{2}=1$, the word $a_{123}a_{124}a_{134}a_{124}a_{134}a_{123}$, cannot be deformed to the word $a_{234}a_{134}a_{124}a_{134}a_{123}a_{234}$.

It is well-known that in Artin presentation of the classical braids group, the representative of the minimal length is not unique. For example, it is related to handle reductions, which play important role in the algorithm Dehornoy~\cite{Dehornoy}. 
      

It is interested to study connections between handle reductions for classical braids and the phenomenon in $G_{4}^{3}$ as the above. 
\end{rem}



\end{document}